\documentclass[10pt]{article}%
\usepackage{a4}%
\usepackage[centertags]{amsmath}%
\usepackage{eucal,dsfont,accents,bbm,amsfonts,amssymb,amsthm,amsopn}%
\usepackage{dsfont}
\usepackage[usenames]{color}
\definecolor{citegreen}{rgb}{0,0.6,0}
\definecolor{refred}{rgb}{0.8,0,0}
\usepackage[colorlinks, citecolor=citegreen, linkcolor=refred]{hyperref}
\title{Monotone volume formulas for geometric flows}%
\author{Reto M\"{u}ller}
\date{}
\providecommand{\abs}[1]{\left\lvert #1\right\rvert}%
\providecommand{\scal}[1]{\left\langle #1\right\rangle}%
\DeclareMathOperator{\Rm}{Rm}%
\DeclareMathOperator{\Ric}{Ric}%
\DeclareMathOperator{\tr}{tr}%
\DeclareMathOperator{\Hess}{Hess}%
\newcommand{\D}{\nabla}%
\newcommand{\Lap}{\triangle}%
\newcommand{\eps}{\varepsilon}%
\newcommand{\RR}{\mathbb{R}}%
\newcommand{\dt}{\mathop{\partial_t}}%
\newcommand{\dte}{\mathop{\partial_{t_1}}\!}%
\newcommand{\ds}{\mathop{\partial_s}}%
\newcommand{\dtau}{\mathop{\partial_\tau}}%
\newcommand{\dl}{\mathop{\partial_\lambda}}%
\newcommand{\sL}{\mathcal{L}}%
\newcommand{\sS}{\mathcal{S}}%
\newcommand{\sH}{\mathcal{H}}%
\newcommand{\sD}{\mathcal{D}}%
\newcommand{\hap}{\mathbin{\hat+}}%
\newcommand{\ham}{\mathbin{\hat-}}%
\newtheoremstyle{break}%
  {12pt}%
  {16pt}%
  {\itshape}%
  {}%
  {\bfseries}%
  {}%
  {\newline}%
  {\thmname{#1}\thmnumber{ #2}\thmnote{ \normalfont{(#3)}}}%
\theoremstyle{break}%
\newtheorem{lemma}{Lemma}[section]%
\newtheorem{theorem}[lemma]{Theorem}%
\newtheorem{definition}[lemma]{Definition}%
\numberwithin{equation}{section}%
\begin{document}%
\maketitle%
\pagenumbering{arabic}%
\vspace{-50mm}
J. reine angew. Math. \textbf{643} (2010), 39--57\\
DOI 10.1515/CRELLE.2010.044\\%
\vspace{35mm}
\begin{abstract}
We consider a closed manifold $M$ with a Riemannian metric
$g_{ij}(t)$ evolving by $\dt g_{ij}=-2S_{ij}$ where $S_{ij}(t)$ is a
symmetric two-tensor on $(M,g(t))$. We prove that if $S_{ij}$
satisfies the tensor inequality $\sD(S_{ij},X)\geq 0$ for all vector
fields $X$ on $M$, where $\sD(S_{ij},X)$ is defined in
(\ref{1.eqD}), then one can construct a forwards and a backwards
reduced volume quantity, the former being non-increasing, the latter
being non-decreasing along the flow $\dt g_{ij}=-2S_{ij}$. In the
case where $S_{ij} = R_{ij}$, the Ricci curvature of $M$, the result
corresponds to Perelman's well-known reduced volume monotonicity for
the Ricci flow presented in \cite{Perelman:entropy}. Some other
examples are given in the second section of this article, the main
examples and motivation for this work being List's extended Ricci
flow system developed in \cite{List:diss}, the Ricci flow coupled
with harmonic map heat flow presented in
\cite{Muller:RicciHarmonic}, and the mean curvature flow in
Lorentzian manifolds with nonnegative sectional curvatures. With our
approach, we find new monotonicity formulas for these flows.
\end{abstract}
\bigskip
%
%
\section{Introduction and formulation of the main result}\label{Sect1}
Let $M$ be a closed manifold with a time-dependent Riemannian metric
$g_{ij}(t)$. Let $\sS(t)$ be a symmetric two tensor on $(M,g(t))$
with components $S_{ij}(t)$ and trace $S(t) := \tr_{g(t)}\sS(t) =
g^{ij}(t)S_{ij}(t)$. Assume that $g(t)$ evolves according to the
flow equation
\begin{equation}\label{1.eq1}
\dt g_{ij}(t) = -2 S_{ij}(t).
\end{equation}
A typical example would be the case where $\sS(t) = \Ric(t)$ is the
Ricci tensor of $(M,g(t))$ and the metric $g(t)$ is a solution to
the Ricci flow, introduced by Richard Hamilton in
\cite{Hamilton:3folds}. Other examples are given in Section \ref{Sect2} of
this article.\\\\%
In analogy to Perelman's $\sL$-distance for the Ricci flow defined
in \cite{Perelman:entropy}, we will now introduce forwards and
backwards reduced distance functions for the flow (\ref{1.eq1}), as
well as a forwards and a backwards reduced volume.
\begin{definition}[forwards reduced distance and volume]\label{1.defF}
Suppose that (\ref{1.eq1}) has a solution for $t \in [0,T]$. For
$0\leq t_0\leq t_1\leq T$ and a curve $\gamma:[t_0,t_1]\to M$, we
define the $\sL_f$-length of $\gamma(t)$ by
\begin{equation*}
\sL_f(\gamma) := \int_{t_0}^{t_1}\sqrt{t}\left(S(\gamma(t)) +
\abs{\dt \gamma(t)}^2 \right)dt.
\end{equation*}
For a fixed point $p\in M$ and $t_0=0$, we define the forwards
reduced distance
\begin{equation}\label{1.eqLf}
\ell_f(q,t_1):= \inf_{\gamma \in
\Gamma}\left\{\frac{1}{2\sqrt{t_1}}\int_0^{t_1}\sqrt{t}
\left(S+\abs{\dt \gamma}^2\right)dt\right\},
\end{equation}
where $\Gamma= \{\gamma:[0,t_1]\to M \mid
\gamma(0)=p,\,\gamma(t_1)=q\}$, i.e.~the forwards reduced distance
is the $\sL_f$-length of an $\sL_f$-shortest curve times
$\frac{1}{2\sqrt{t_1}}$. Existence of such $\sL_f$-shortest curves
will be discussed in the fourth section. Finally, the forwards
reduced volume is defined to be
\begin{equation}\label{1.eqVf}
V_f(t) := \int_M (4\pi t)^{-n/2}e^{\ell_f(q,t)} dV(q).
\end{equation}
\end{definition}
In order to define the backwards reduced distance and volume, we
need a backwards time $\tau(t)$ with $\dt \tau(t) = -1$. Without
loss of generality, one may assume (possibly after a time shift)
that $\tau = -t$.
\begin{definition}[backwards reduced distance and volume]\label{1.defB}
If (\ref{1.eq1}) has a solution for $\tau \in [0,\bar{\tau}]$ we
define the $\sL_b$-length of a curve $\gamma:[\tau_0,\tau_1]\to M$
by
\begin{equation*}
\sL_b(\gamma) :=
\int_{\tau_0}^{\tau_1}\sqrt{\tau}\left(S(\gamma(\tau)) + \abs{\dtau
\gamma(\tau)}^2 \right)d\tau.
\end{equation*}
Again, we fix the point $p\in M$ and $\tau_0=0$ and define the
backwards reduced distance by
\begin{equation}\label{1.eqLb}
\ell_b(q,\tau_1):= \inf_{\gamma \in
\Gamma}\left\{\frac{1}{2\sqrt{\tau_1}}\int_0^{\tau_1}\sqrt{\tau}
\left(S+\abs{\dtau \gamma}^2\right)d\tau\right\},
\end{equation}
where now $\Gamma= \{\gamma:[0,\tau_1]\to M \mid
\gamma(0)=p,\,\gamma(\tau_1)=q\}$. The backwards reduced volume is
defined by
\begin{equation}\label{1.eqVb}
V_b(\tau) := \int_M (4\pi\tau)^{-n/2}e^{-\ell_b(q,\tau)}dV(q).
\end{equation}
\end{definition}
Next, we define an evolving tensor quantity $\sD$ associated to the
tensor $\sS$.
\begin{definition}\label{1.defD}
Let $g(t)$ evolve by $\dt g_{ij}=-2S_{ij}$ and let $S$ be the trace
of $\sS$ as above. Let $X \in \Gamma(TM)$ be a vector field on $M$.
We set
\begin{equation}\label{1.eqD}
\begin{split}
\sD(\sS,X) &:= \dt S-\Lap S -2\abs{S_{ij}}^2
+4(\D_i S_{ij})X_j -2(\D_j S)X_j\\%
&\phantom{:}\quad + 2R_{ij}X_iX_j - 2S_{ij}X_iX_j,
\end{split}
\end{equation}
\end{definition}
\emph{Remark.} The quantity $\sD$ consists of three terms. The first
term, $\dt S-\Lap S -2\abs{S_{ij}}^2$, captures the evolution
properties of $S = g^{ij}S_{ij}$ under the flow (\ref{1.eq1}). The
second one, $4(\D_i S_{ij})X_j - 2(\D_j S)X_j$, is a multiple of the
error term $E$ that appears in the twice traced second Bianchi type
identity $\D_i S_{ij} = \tfrac{1}{2}\D_j S + E$ for the symmetric
tensor $\sS$. Finally, the last term directly compares the tensor
$S_{ij}$ with the Ricci tensor.\\\\%
We can now state our main result.
\begin{theorem}[monotonicity of forwards and backwards reduced volume]\label{1.mainthm}%
Suppose that $g(t)$ evolves by (\ref{1.eq1}) and the quantity
$\sD(\sS,X)$ is nonnegative for all vector fields $X \in \Gamma(TM)$
and all times $t$ for which the flow exists. Then the
\emph{forwards} reduced volume $V_f(t)$ is non-increasing in $t$
along the flow. Moreover, the \emph{backwards} reduced volume
$V_b(\tau)$ is non-increasing in $\tau$, i.e.~non-decreasing in $t$.
\end{theorem}
The remainder of this work is organized as follows. In the next
section, we consider some examples where Theorem \ref{1.mainthm} can
be applied. In Section \ref{Sect3}, we start the proof of the theorem by
showing that the quantity $\sD(\sS,X)$ is the difference between two
differential Harnack type quantities for the tensor $\sS$ defined as
follows.
\begin{definition}\label{1.defH}
For two tangent vector fields $X,Y \in \Gamma(TM)$ on $M$, we define
\begin{align*}
\sH(\sS,X,Y)&:= 2(\dt \sS)(Y,Y)+2\abs{\sS(Y,\cdot)}^2-\D_Y\D_YS +
\tfrac{1}{t}\sS(Y,Y)\\%
&\quad -4(\D_X\sS)(Y,Y)+4(\D_Y\sS)(X,Y)-2\scal{\Rm(X,Y)X,Y},\\%
\sH(\sS,X)&:= \dt S + \tfrac{1}{t}S - 2\scal{\D S,X} + 2\sS(X,X).
\end{align*}
\end{definition}
\begin{lemma}\label{1.lemmaD}
The quantity $\sD(\sS,X)$ is the difference between the trace of
$\sH(\sS,X,Y)$ with respect to the vector field $Y$ and the
expression $\sH(\sS,X)$, i.e.~for an orthonormal basis $\{e_i\}$, we
have
\begin{equation*}
\sD(\sS,X) = \sum_i \sH(\sS,X,e_i) - \sH(\sS,X).
\end{equation*}
\end{lemma}
In Section \ref{Sect4}, after introducing a notation which makes it possible
to deal with the forwards and the backwards case at the same time,
we study geodesics for the $\sL$-length functionals and some
regularity properties for the corresponding distances. In the last
section finally, we prove Theorem \ref{1.mainthm}, following the
proof for the Ricci flow case by Perelman from
\cite{Perelman:entropy}, Section 7. See also Kleiner and Lott
\cite{KleinerLott} or M\"{u}ller \cite{Muller:Harnack} for more details.
\paragraph{Acknowledgements:}
I thank Gerhard Huisken and Klaus Ecker for their invitation for a
four month research visit in Potsdam and Berlin. During this time I
developed the $\sL_b$-functional and the monotonicity of the
backwards reduced volume for the case of List's extended Ricci flow
system \cite{List:diss} discussed in Section \ref{Sect2} in joint work with
Valentina Vulcanov. This was part of Valentina Vulcanov's master
thesis \cite{Vulcanov:master}. The present work is a natural
generalization of this result. Last but not least, I also thank
Michael Struwe for valuable suggestions and the Swiss National
Science Foundation for financial support.
%
%
\section{Some examples}\label{Sect2}
\paragraph{i) The static case.}
Let $(M,g)$ be a Riemannian manifold and set $S_{ij}=0$ so that $g$
is fixed. Then the quantity $\sD$ reduces to $\sD(0,X)=2R_{ij}X_iX_j
= 2\Ric(X,X)$. In the case where $M$ has nonnegative Ricci
curvature, i.e.~$\sD(0,X)\geq 0$ for all vector fields $X$ on $M$,
Theorem \ref{1.mainthm} can be applied. For example the
\emph{backwards} reduced volume
\begin{equation*}
V_b(\tau) = \int_M (4\pi \tau)^{-n/2}e^{-\ell_b(q,\tau)} dV
\end{equation*}
is non-increasing in $\tau$, where
\begin{equation*}
\ell_b(q,\tau_1):= \inf_{\gamma \in
\Gamma}\left\{\frac{1}{2\sqrt{\tau_1}}\int_0^{\tau_1}\sqrt{\tau}
\abs{\dtau \gamma}^2 d\tau\right\}.
\end{equation*}
Note that the assumption $\Ric \geq 0$ is necessary for the
monotonicity, a result which we already proved in
\cite{Muller:Harnack}, page 72.%
\paragraph{ii) The Ricci flow.}
Let $(M,g(t))$ be a solution to the Ricci flow, i.e.~let
$S_{ij}=R_{ij}$ be the Ricci and $S=R$ the scalar curvature tensor
on $M$. Since the scalar curvature evolves by $\dt R = \Lap R +
2\abs{R_{ij}}^2$, and because of the twice traced second Bianchi
identity $\D_i R_{ij}=\tfrac{1}{2}\D_j R$, we see from (\ref{1.eqD})
that the quantity $\sD(\Ric,X)$ vanishes identically on $M$. Hence
the theorem can be applied. Note that $\sH(\Ric,X,Y)$ and
$\sH(\Ric,X)$ denote Hamilton's matrix and trace Harnack quantities
for the Ricci flow from \cite{Hamilton:RicciHarnack}. The backwards
reduced volume corresponds to the one defined by Perelman in
\cite{Perelman:entropy}, the forwards reduced volume and the proof
of its monotonicity were developed by Feldman, Ilmanen
and Ni in \cite{FeldmanIlmanenNi}.%
\paragraph{iii) Bernhard List's flow.}
In his dissertation \cite{List:diss}, Bernhard List introduced a
system closely related to the Ricci flow, namely
\begin{equation}\label{2.List1}
\left\{\begin{split}\dt g &= -2\Ric{} + 4 \; \D\psi \otimes \D\psi,\\%
\dt \psi &= \Lap_g \psi, \end{split}\right.
\end{equation}
where $\psi:M\to\RR$ is a smooth function. His motivation came from
general relativity theory: for static vacuum solutions, the Einstein
evolution problem -- which is in general a hyperbolic system of
partial differential equations describing a Lorentzian 4-manifold --
reduces to a weakly elliptic system on a 3-dimensional Riemannian
manifold $M$, the space slice in the so-called $3+1$ split of
space-time (cf. \cite{List:diss}, \cite{MisnerThorneWheeler}). The
remaining freedom for solutions consists of the Riemannian metric
$g$ on $M$ and the lapse function, which measures the speed of the
space slice in time direction. If we let $\psi$ be the logarithm of
the lapse function, the static Einstein vacuum equations read
\begin{equation}\label{2.List2}
\left\{\begin{split}\Ric(g) &= 2 \; \D\psi \otimes \D\psi,\\%
\Lap_g \psi &= 0. \end{split}\right.
\end{equation}
Clearly the solutions of (\ref{2.List2}) are exactly the stationary
solutions of (\ref{2.List1}).\\\\%
If we set $S_{ij}=R_{ij}-2\D_i\psi\D_j\psi$ with
$S=R-2\abs{\D\psi}^2$, the first of the flow equations in
(\ref{2.List1}) is again of the form $\dt g_{ij}=-2S_{ij}$. List
proved (\cite{List:diss}, Lemma 2.11) that under this flow
\begin{equation*}
\dt S = \Lap S +2\abs{S_{ij}}^2 + 4\abs{\Lap \psi}^2.
\end{equation*}
Moreover, a direct computation shows that
\begin{equation*}
4(\D_i S_{ij})-2(\D_j S) = -8\D_i(\D_i\psi\D_j\psi)
+4\D_j(\D_i\psi\D_i\psi)=-8\Lap\psi\D_j\psi,
\end{equation*}
and plugging this into (\ref{1.eqD}) yields
\begin{equation}
\sD(S_{ij},X)=4\abs{\Lap\psi}^2-8\Lap\psi\D_j\psi X_j +
4\D_i\psi\D_j\psi X_iX_j = 4\abs{\Lap\psi-\D_X\psi}^2 \geq 0
\end{equation}
for all vector fields $X$ on $M$. Hence, we can apply the main
theorem, i.e.~the backwards and forwards reduced volume monotonicity
results hold for List's flow.%
\paragraph{iv) The Ricci flow coupled with harmonic map heat flow.}
This flow is introduced in \cite{Muller:RicciHarmonic}. Let $M$ be
closed and fix a Riemannian manifold $(N,\gamma)$. The couple
$(g(t),\phi(t))_{t \in [0,T)}$ consisting of a family of smooth
metrics $g(t)$ on $M$ and a family of smooth maps $\phi(t)$ from $M$
to $N$ is called a solution to the Ricci flow coupled with harmonic
map heat flow with coupling function $\alpha(t) \geq 0$, if it
satisfies
\begin{equation}\label{2.RH1}
\left\{\begin{split}\dt g &= -2\Ric{} + 2\alpha \; \D\phi \otimes \D\phi,\\%
\dt \phi &= \tau_g \phi, \end{split}\right.
\end{equation}
Here, $\tau_g \phi$ denotes the tension field of the map $\phi$ with
respect to the evolving metric $g$. Note that Bernhard List's flow
above corresponds to the special case where $\alpha=2$ and $N=\RR$
(and thus $\tau_g = \Lap_g$ is the Laplace-Beltrami operator). We
now show that the monotonicity of the reduced volumes holds for this
more general flow. To this end, we set $S_{ij}=R_{ij}- \alpha
\D_i\phi\D_j\phi$ with trace $S=R-2\alpha \, e(\phi)$, where
$e(\phi) = \frac{1}{2}\abs{\D \phi}^2$ denotes the standard local
energy density of the map $\phi$. In \cite{Muller:RicciHarmonic}, we
prove the evolution equation
\begin{equation*}
\dt S = \Lap S + 2\abs{S_{ij}} + 2\alpha \abs{\tau_g \phi}^2
-2\dot{\alpha}e(\phi).
\end{equation*}
Using $4(\D_i S_{ij})X_j - 2(\D_j S)X_j = -4\alpha\, \tau_g \phi
\D_j\phi X_j$ and plugging into (\ref{1.eqD}), we get
\begin{equation}
\sD(S_{ij},X)=2\alpha \abs{\tau_g\phi-\D_X\phi}^2
-2\dot{\alpha}e(\phi)
\end{equation}
for all $X$ on $M$. Thus, we can again apply Theorem \ref{1.mainthm}
if $\alpha(t)\geq 0$ is non-increasing.
\paragraph{v) The mean curvature flow.}
Let $M^n(t) \subset \RR^{n+1}$ denote a family of hypersurfaces
evolving by mean curvature flow. Then the induced metrics evolve by
$\dt g_{ij} = -2HA_{ij}$, where $A_{ij}$ denote the components of
the second fundamental form $A$ on $M$ and $H=g^{ij}A_{ij}$ denotes
the mean curvature of $M$. Letting $S_{ij}=HA_{ij}$ with trace
$S=H^2$, the expression $\sH(\sS,X)$ from Definition \ref{1.defH}
becomes
\begin{align*}
\sH(\sS,X)&=\dt H^2 + \tfrac{1}{t}H^2 -
2\big\langle\D H^2,X\big\rangle + 2HA(X,X)\\%
&= 2H\big(\dt H + \tfrac{1}{2t}H - 2\scal{\D H,X} + A(X,X)\big),
\end{align*}
that is $2H$ times Hamilton's differential Harnack expression for
the mean curvature flow defined in \cite{Hamilton:McfHarnack}.
Moreover, the quantity $\sD(\sS,X)$ again has a sign for all vector
fields $X$, but unfortunately the wrong one for our purpose. Indeed,
one finds $\sD(\sS,X)= -2\abs{\D H-A(X,\cdot)}^2 \leq 0$, $\forall
X\in\Gamma(TM)$, and Theorem \ref{1.mainthm} can't be applied. But
fortunately the sign changes if we consider mean curvature flow in
Minkowski space, as suggested by Mu-Tao Wang at a conference in
Oberwolfach. More general, let $M^n(t) \subset L^{n+1}$ be a family
of spacelike hypersurfaces in an ambient Lorentzian manifold,
evolving by Lorentzian mean curvature flow. Then the induced metric
solves $\dt g_{ij}=2HA_{ij}$, i.e.~we have $S_{ij}=-HA_{ij}$ and
$S=-H^2$. Marking the curvature with respect to the ambient manifold
with a bar, we have the Gauss equation
\begin{equation*}
R_{ij}=\bar{R}_{ij}-HA_{ij}+A_{i\ell}A_{\ell j} + \bar{R}_{i0j0},
\end{equation*}
the Codazzi equation
\begin{equation*}
\D_iA_{jk}-\D_kA_{ij}=\bar{R}_{0jki},
\end{equation*}
as well as the evolution equation for the mean curvature
\begin{equation*}
\dt H = \Lap H-H(\abs{A}^2 +\overline{\Ric}(\nu,\nu)),
\end{equation*}
cf. Section 2.1 and 4.1 of Holder \cite{Holder:diss}. Here, $\nu$
denotes the future-oriented timelike normal vector, represented by
$0$ in the index-notation. Combining the three equations above, we
find
\begin{equation}
\sD(\sS,X)=2\abs{\D H-A(X,\cdot)}^2 +
2\overline{\Ric}(H\nu-X,H\nu-X) +
2\scal{\overline{\Rm}(X,\nu)\nu,X}.
\end{equation}
In particular, if $L^{n+1}$ has nonnegative sectional curvatures, we
get $\sD(\sS,X)\geq 0$ and our main theorem can be applied.
%
%
\section{Proof of Lemma \ref{1.lemmaD}}\label{Sect3}
This is just a short computation. First, note that since the metric
evolves by $\dt g_{ij} = -2S_{ij}$ its inverse evolves by $\dt
g^{ij} = 2S^{ij} := 2g^{ik}g^{j\ell}S_{k\ell}$. As a consequence
\begin{equation}
\dt S = \dt (g^{ij}S_{ij}) = 2 \abs{S_{ij}}^2 + \sum_i (\dt
\sS)(e_i,e_i),
\end{equation}
where $\{e_i\}$ is an orthonormal basis. Therefore, by tracing and
rearranging the terms, we find
\begin{align*}
\sum_i \sH(\sS,X,e_i) &= \sum_i\big(2(\dt \sS)(e_i,e_i) + 2
\abs{\sS(e_i,\cdot)}^2 - \D_{e_i}\D_{e_i}S +
\tfrac{1}{t}\sS(e_i,e_i)\big)\\%
&\quad +\sum_i \big(- 4(\D_X\sS)(e_i,e_i) + 4 (\D_{e_i}\sS)(X,e_i) -
2\scal{\Rm(X,e_i)X,e_i}\big)\\%
&= 2\big(\dt S - 2\abs{S_{ij}}^2\big) + 2\abs{S_{ij}}^2 -\Lap S +
\tfrac{1}{t}S\\%
&\quad -4(\D_j S)X_j + 4(\D_iS_{ij})X_j + 2\Ric(X,X)\\%
&= \dt S -2\abs{S_{ij}}^2-\Lap S - 2(\D_jS)X_j + 4(\D_iS_{ij})X_j +
2R_{ij}X_iX_j\\%
&\quad - 2S_{ij}X_iX_j + \dt S + \tfrac{1}{t}S -2(\D_jS)X_j +
2S_{ij}X_iX_j\\%
&= \sD(\sS,X) + \sH(\sS,X).
\end{align*}
This proves the lemma.\qed
%
%
\section{$\sL_f$-geodesics and $\sL_b$-geodesics}\label{Sect4}%
Obviously, letting $\tau$ play the role of the forwards time, the
backwards reduced distance as defined in Definition \ref{1.defB}
corresponds to the forwards reduced distance for the flow $\dtau
g_{ij} = + 2S_{ij}$. Thus the computations in the forwards and the
backwards case differ only by the change of some signs and we find
it convenient to do them only for the forwards case. However, we
mark all the signs that change in the backwards case with a hat. We
illustrate this with an example. Equation (\ref{5.eqNotation}) below
reads
\begin{equation*}
t^{3/2}\tfrac{d}{dt} \big(S+\abs{X}^2 \big)= \hap
t^{3/2}\sH(\sS,\ham X) - \sqrt{t}\big(S + \abs{X}^2\big),
\end{equation*}
with $\sH(\sS,\ham X)$ evaluated at time $t$. For the forwards case,
we simply neglect the hats and interpret $\sH(\sS,-X)$ as in
Definition \ref{1.defH}. For the backwards case, we change all $t$,
$\dt$ into $\tau$, $\dtau$ etc. and change all the signs with a hat,
i.e.~the statement is
\begin{equation*}
\tau^{3/2}\tfrac{d}{d\tau} \big(S+\abs{X}^2 \big)= -
\tau^{3/2}\sH(\sS,X) - \sqrt{\tau}\big(S + \abs{X}^2\big),
\end{equation*}
where $\sH(\sS,X)$ is now evaluated at $\tau=-t$, i.e.
\begin{equation}\label{4.defH1}
\sH(\sS,X)= -\dtau S - \tfrac{1}{\tau}S -2\scal{\D S,X}+2\sS(X,X).
\end{equation}
Similarly, the matrix Harnack type expression $\sH(\sS,X,Y)$ from
Definition \ref{1.defH} has to be interpreted as
\begin{equation}\label{4.defH2}
\begin{split}
\sH(\sS,X,Y)&= -2(\dtau \sS)(Y,Y)+2\abs{\sS(Y,\cdot)}^2-\D_Y\D_YS
-\tfrac{1}{\tau}\sS(Y,Y)\\%
&\quad -4(\D_X\sS)(Y,Y)+4(\D_Y\sS)(X,Y)-2\scal{\Rm(X,Y)X,Y}\\%
\end{split}
\end{equation}
in the backwards case.\\\\%
For the Ricci flow there exist various references where the
following computations can be found in detail for the backwards
case, for example Kleiner and Lott \cite{KleinerLott}, M\"{u}ller
\cite{Muller:Harnack} or Chow et al.~\cite{RF:TAI}. The forwards
case for the Ricci flow can be found in Feldman, Ilmanen and Ni
\cite{FeldmanIlmanenNi}. This and the following section follow
these sources closely.%
\paragraph{The geodesic equation.}%
Let $0 < t_0\leq t_1 \leq T$ and let $\gamma_s(t)$ be a variation of
the path $\gamma(t): [t_0,t_1]\to M$. Using Perelman's notation, we
set $Y(t)=\ds \gamma_s(t)|_{s=0}$ and $X(t)=\dt \gamma_s(t)|_{s=0}$.
The first variation of $\sL_f(\gamma)$ in the direction of $Y(t)$
can then be computed as follows.
\begin{align*}
\delta_Y \sL_f(\gamma) &:= \ds \sL_f(\gamma_s)|_{s=0} =
\int_{t_0}^{t_1}\sqrt{t}\ds\big(S(\gamma_s(t))+
\scal{\dt \gamma_s,\dt \gamma_s}\big)|_{s=0} \;dt\\%
&= \int_{t_0}^{t_1} \sqrt{t}\big(\D_Y S +2\scal{\D_Y X,X}\big)dt =
\int_{t_0}^{t_1} \sqrt{t}\big(\D_Y S +2\scal{\D_X
Y,X}\big)dt\\%
&= \int_{t_0}^{t_1} \sqrt{t}\big(\scal{Y,\D
S}+2\dt\scal{Y,X}-2\scal{Y,\D_X X}\hap 4 \sS(Y,X)\big)dt\\%
&= 2\sqrt{t}\scal{Y,X}\big|_{t_0}^{t_1} + \int_{t_0}^{t_1}
\sqrt{t}\scal{Y,\D S - \tfrac{1}{t}X -2\D_X X \hap 4\sS(X,\cdot)}dt,
\end{align*}
using a partial integration in the last step. An
$\sL_f$-\emph{geodesic} is a critical point of the $\sL_f$-length
with respect to variations with fixed endpoints. Hence, the above
first variation formula implies that the $\sL_f$-geodesic equation
reads
\begin{equation}\label{4.eqGeo1}
G_f(X) := \D_X X - \tfrac{1}{2}\D S + \tfrac{1}{2t}X \ham
2\sS(X,\cdot) = 0.
\end{equation}
Changing the variable $\lambda=\sqrt{t}$ in the definition of
$\sL_f$-length, we get
\begin{equation*}
\sL_f(\gamma(\lambda)) = \int_{\lambda_0}^{\lambda_1}
\big(2\lambda^2 S(\gamma(\lambda)) + \tfrac{1}{2}\abs{\dl
\gamma(\lambda)}^2\big) d\lambda,
\end{equation*}
and the Euler-Lagrange equation (\ref{4.eqGeo1}) becomes
\begin{equation}\label{4.eqGeo2}
G_f(\tilde{X}) := \D_{\tilde{X}}\tilde{X}-2\lambda^2\D S \ham
4\lambda \sS(\tilde{X},\cdot) =0,
\end{equation}
where $\tilde{X}=\dl \gamma(\lambda)=2\lambda X$.%
\paragraph{Existence of $\sL_f$-geodesics.}%
From standard existence theory for ordinary differential equations,
we see that for $\lambda_0=\sqrt{t_0}$, $p\in M$ and $v\in T_pM$
there is a unique solution $\gamma(\lambda)$ to (\ref{4.eqGeo2}) on
an interval $[\lambda_0,\lambda_0+\eps]$ with $\gamma(\lambda_0)=p$
and $\dl\gamma(\lambda)|_{\lambda=\lambda_0}=\lim_{t\to
t_0}2\sqrt{t}X =v$. If $C$ is a bound for $\abs{\sS}$ and $\abs{\D
S}$ on $M \times [0,T]$ and $\tilde{X}(\lambda)\neq 0$, we find for
$\sL_f$-geodesics
\begin{equation}\label{4.eqGrowth}
\begin{split}
\dl\lvert\tilde{X}\rvert = \tfrac{1}{2\lvert\tilde{X}\rvert}\dl
\lvert\tilde{X}\rvert^2 &= \hap 2\lambda \lvert\tilde{X}\rvert \;
\sS\Big(\tfrac{\tilde{X}}{\lvert\tilde{X}\rvert},
\tfrac{\tilde{X}}{\lvert\tilde{X}\rvert}\Big) + 2\lambda^2 \scal{
\D S, \tfrac{\tilde{X}}{\lvert\tilde{X}\rvert}}\\%
&\leq 2\lambda C\lvert\tilde{X}\rvert + 2\lambda^2 C .
\end{split}
\end{equation}
Hence, by a continuity argument, the unique $\sL_f$-geodesic
$\gamma(\lambda)$ can be extended to the whole interval
$[\lambda_0,\sqrt{T}]$, i.e.~for any $p\in M$ and $t_1 \in[t_0,T]$
we get a globally defined smooth $\sL_f$-exponential map, taking
$v\in T_pM$ to $\gamma(t_1)$, where $\lim_{t\to t_0}2\sqrt{t}\dt
\gamma(t) =v$. Moreover, $\tilde{X}=2\sqrt{t}X(t)$ has a limit as
$t\to 0$ for $\sL_f$-geodesics and the definition of $\sL_f(\gamma)$
can be extended to $t_0=0$.\\\\%
For all $(q,t_1)$ there exists a minimizing $\sL_f$-geodesic from
$p=\gamma(0)$ to $q=\gamma(t_1)$. To see this, we can either show
that $\sL_f$-geodesics minimize for a short time and then use the
broken geodesic argument as in the standard Riemannian case, or
alternatively we can use the direct method of calculus of
variations. There exists a minimizer of $\sL_f(\gamma)$ among all
Sobolev curves, which then has to be a solution of (\ref{4.eqGeo1})
and hence a smooth $\sL_f$-geodesic.\\\\%
In the following, we fix $p\in M$ and $t_0=0$ and denote by
$L_f(q,t_1)$ the $\sL_f$-length of a shortest $\sL_f$-geodesic
$\gamma(t)$ joining $p=\gamma(0)$ with $q=\gamma(t_1)$, i.e.~the
reduced length is
\begin{equation*}
\ell_f(q,t_1)=\tfrac{1}{2\sqrt{t_1}}L_f(q,t_1).
\end{equation*}
\paragraph{Technical issues about $L_f(q,t_1)$.}%
We first prove lower and upper bounds for $L_f(q,t_1)$. Since $M$ is
closed, there is a positive constant $C_0$ such that $-C_0g(t) \leq
\sS(t) \leq C_0g(t)$ (and thus $-C_0n \leq S(t) \leq C_0n$) for all
$t \in [0,T]$. We can then obtain the following estimates.
\begin{lemma}\label{4.lemmaBounds}
Denote by $d(p,q)$ the standard distance between $p$ and $q$ at time
$t=0$, i.e. the Riemannian distance with respect to $g(0)$. Then the
reduced distance $L_f(q,t_1)$ satisfies
\begin{equation}\label{4.eqBounds}
\frac{d^2(p,q)}{2\sqrt{t_1}}e^{-2C_0t_1} -\frac{2nC_0}{3}t_1^{3/2}
\leq L_f(q,t_1) \leq \frac{d^2(p,q)}{2\sqrt{t_1}}e^{2C_0t_1}
+\frac{2nC_0}{3}t_1^{3/2}.
\end{equation}
\end{lemma}
\begin{proof}
The bounds for $\sS(t)$ imply $-2C_0g(t) \leq \ham 2\sS(t) = \dt
g(t) \leq 2C_0g(t)$ and thus
\begin{equation}\label{4.eq7}
e^{-2C_0t}g(0) \leq g(t) \leq e^{2C_0t}g(0).
\end{equation}
Using $\lambda = \sqrt{t}$ as above, we can estimate
\begin{align*}
\sL_f(\gamma) &= \int_0^{\sqrt{t_1}} \Big(\tfrac{1}{2}\abs{\dl
\gamma(\lambda)}^2 + 2\lambda^2 S(\gamma(\lambda))\Big)
d\lambda\\%
&\geq \tfrac{1}{2}e^{-2C_0t_1}\int_0^{\sqrt{t_1}} \abs{\dl
\gamma(\lambda)}^2_{g(0)} d\lambda -\tfrac{2}{3}nC_0\lambda^3
\big|_0^{\sqrt{t_1}}\\%
&\geq \frac{d^2(p,q)}{2\sqrt{t_1}}e^{-2C_0t_1} - \frac{2nC_0}{3}\,
t_1^{3/2}.
\end{align*}
With $L_f(q,t_1)= \inf_{\gamma \in \Gamma} \sL_f(\gamma)$ we get the
lower bound in (\ref{4.eqBounds}). For the upper bound, let
$\eta(\lambda):[0,\sqrt{t_1}]\to M$ be a minimal geodesic from $p$
to $q$ with respect to $g(0)$. Then
\begin{align*}
L_f(q,t_1) &\leq \sL_f(\eta) = \int_0^{\sqrt{t_1}}
\Big(\tfrac{1}{2}\abs{\dl \eta(\lambda)}^2 + 2\lambda^2
S(\eta(\lambda))\Big)
d\lambda\\%
&\leq \tfrac{1}{2}e^{2C_0t_1}\int_0^{\sqrt{t_1}} \abs{\dl
\eta(\lambda)}^2_{g(0)} d\lambda +\tfrac{2}{3}nC_0\lambda^3
\big|_0^{\sqrt{t_1}}\\%
&= \frac{d^2(p,q)}{2\sqrt{t_1}}e^{2C_0t_1} + \frac{2nC_0}{3}\,
t_1^{3/2},
\end{align*}
which proves the claim.
\end{proof}
\begin{lemma}\label{4.lemmaLip}%
The distance $L_f:M\times(0,T)\to \RR$ is locally Lipschitz
continuous with respect to the metric $g(t)+dt^2$ on space-time and
smooth outside of a set of measure zero.
\end{lemma}
\begin{proof}
For any $0 < t_* < T$, $q_* \in M$ and small $\eps > 0$, let $t_1 <
t_2$ be in $(t_*-\eps,t_*+\eps)$ and $q_1,q_2 \in
B_{g(t_*)}(q_*,\eps) = \{q\in M \mid d_{g(t_*)}(q_*,q) < \eps\}$,
where $d_{g(t_*)}(\cdot,\cdot)$ denotes the Riemannian distance with
respect to the metric $g(t_*)$. Since
\begin{equation*}
\abs{L_f(q_1,t_1)-L_f(q_2,t_2)} \leq
\abs{L_f(q_1,t_1)-L_f(q_1,t_2)}+\abs{L_f(q_1,t_2)-L_f(q_2,t_2)},
\end{equation*}
it suffices for the Lipschitz continuity with respect to $g(t)+dt^2$
to show that $L_f(q_1,\cdot)$ is locally Lipschitz in the time
variable uniformly in $q_1 \in B_{g(t_*)}(q_*,\eps)$ and
$L_f(\cdot,t)$ is locally Lipschitz in the space variable uniformly
in $t \in (t_*-\eps,t_*+\eps)$. Our proof is related to the proofs
of Lemma 7.28 and Lemma 7.30 in \cite{RF:TAI}. In the following,
$C=C(C_0,n,t_*,\eps)$ denotes a generic constant which might change
from line to line.
\paragraph{Claim 1:} $L_f(q_1,t_2) \leq L_f(q_1,t_1) + C(t_2-t_1)$.
\begin{proof}
Let $\gamma:[0,t_1]\to M$ be a minimal $\sL_f$-geodesic from $p$ to
$q_1$ and define $\eta:[0,t_2]\to M$ by
\begin{equation}\label{5.eqEta1}
\eta(t) := \begin{cases}\gamma(t) &\textrm{ if }t\in [0,t_1],\\%
q_1 &\textrm{ if }t\in [t_1,t_2].\end{cases}
\end{equation}
We compute
\begin{align*}
L_f(q_1,t_2)&\leq \sL_f(\eta)= \sL_f(\gamma) +
\int_{t_1}^{t_2}\sqrt{t}S(q_1,t) dt\\%
&\leq L_f(q_1,t_1) + \tfrac{2}{3}nC_0\Big(t_2^{3/2}-t_1^{3/2}\Big)\\%
&\leq L_f(q_1,t_1) + C(t_2-t_1),
\end{align*}
which proves Claim 1.
\end{proof}
\paragraph{Claim 2:} $L_f(q_1,t_1) \leq L_f(q_1,t_2) + C(t_2-t_1)$.
\begin{proof}
Let $\gamma:[0,t_2]\to M$ be a minimal $\sL_f$-geodesic from $p$ to
$q_1$ and define $\eta:[0,t_1]\to M$ by
\begin{equation}\label{5.eqEta2}
\eta(t) := \begin{cases}\gamma(t) &\textrm{ if }t\in [0,2t_1-t_2],\\%
\gamma(\phi(t)) &\textrm{ if }t\in [2t_1-t_2,t_1],\end{cases}
\end{equation}
where $\phi(t) := 2t + t_2-2t_1 \geq t$ on $[2t_1-t_2,t_1]$ with
$\dt \phi(t) \equiv 2$. We compute
\begin{align*}
L_f(q_1,t_1) &\leq \sL_f(\eta) = \sL_f(\gamma) -
\int_{2t_1-t_2}^{t_2}\sqrt{t}\Big(S(\gamma(t),t)
+\abs{\dt\gamma(t)}^2\Big)dt\\%
&\quad + \int_{2t_1-t_2}^{t_1}\sqrt{t}\Big(S(\gamma(\phi(t)),t)
+\abs{\dt\gamma(\phi(t))\cdot \dt\phi(t)}^2\Big)dt\\%
&\leq L_f(q_1,t_2) + \tfrac{2}{3}nC_0\Big(t_2^{3/2}-(2t_1-t_2)^{3/2}
\Big)+\tfrac{2}{3}nC_0\Big(t_1^{3/2}-(2t_1-t_2)^{3/2}\Big)\\%
&\quad +2 \int_{2t_1-t_2}^{t_2}\sqrt{\phi^{-1}(t)}
\abs{\dt\gamma(t)}^2_{g(\phi^{-1}(t))}dt\\%
&\leq L_f(q_1,t_1) + C(t_2-t_1) + 2\int_{2t_1-t_2}^{t_2}
\sqrt{\phi^{-1}(t)}\abs{\dt\gamma(t)}^2_{g(\phi^{-1}(t))}dt.
\end{align*}
Since $\phi^{-1}(t) \leq t$ and $t-\phi^{-1}(t)\leq 2\eps$ on
$[2t_1-t_2,t_2]$, we can estimate the very last term via
(\ref{4.eq7}) by
\begin{equation*}
\int_{2t_1-t_2}^{t_2} \sqrt{\phi^{-1}(t)}
\abs{\dt\gamma(t)}^2_{g(\phi^{-1}(t))}dt \leq e^{4C_0\eps}
\int_{2t_1-t_2}^{t_2}\sqrt{t}\abs{\dt\gamma(t)}^2_{g(t)}dt.
\end{equation*}
As a consequence of the upper bound from Lemma \ref{4.lemmaBounds}
and the growth condition (\ref{4.eqGrowth}), $\abs{\dt
\gamma(t)}^2_{g(t)}$ must be uniformly bounded on $[2t_1-t_2,t_2]$
by a constant $C_1$. Thus
\begin{equation*}
\int_{2t_1-t_2}^{t_2} \sqrt{\phi^{-1}(t)}
\abs{\dt\gamma(t)}^2_{g(\phi^{-1}(t))}dt \leq
e^{4C_0\eps}C_1\big(t_2^{3/2}-(2t_1-t_2)^{3/2}\big) \leq C(t_2-t_1).
\end{equation*}
Together with the computation above, this proves the claim.
\end{proof}
\paragraph{Claim 3:} $L_f(q_1,t_2) \leq L_f(q_2,t_2) + C d_{g(t_2)}(q_1,q_2)$.
\begin{proof}
Let $\gamma:[0,t_2]\to M$ be a minimal $\sL_f$-geodesic from $p$ to
$q_2$ and define the curve $\eta:[0,t_2+d_{g(t_2)}(q_1,q_2)]\to M$
by
\begin{equation}\label{5.eqEta3}
\eta(t) := \begin{cases}\gamma(t) &\textrm{ if }t\in [0,t_2],\\%
\alpha(t) &\textrm{ if }t\in [t_2,t_2+d_{g(t_2)}(q_1,q_2)],\end{cases}%
\end{equation}
where $\alpha:[t_2,t_2+d_{g(t_2)}(q_1,q_2)]\to M$ is a minimal
geodesic of constant unit speed with respect to $g(t_2)$, joining
$q_2$ to $q_1$. Then, using $\abs{\dt\alpha(t)}^2_{g(t)} \leq
e^{4C_0\eps}\abs{\dt\alpha(t)}^2_{g(t_2)} = e^{4C_0\eps}$, we obtain
\begin{align*}
L_f(q_1,t_2+d_{g(t_2)}(q_1,q_2)) &\leq \sL_f(\eta)\\%
&= L_f(q_2,t_2) + \int_{t_2}^{t_2+d_{g(t_2)}(q_1,q_2)}
\sqrt{t}\Big(S(\alpha(t),t)+\abs{\dt\alpha(t)}^2\Big)dt\\%
&\leq L_f(q_2,t_2) + \tfrac{2}{3}\Big(C_0n + e^{4C_0\eps}\Big)
\Big((t_2+d_{g(t_2)}(q_1,q_2))^{3/2}-t_2^{3/2}\Big)\\%
&\leq L_f(q_2,t_2) + Cd_{g(t_2)}(q_1,q_2).
\end{align*}
Finally, using Claim 2 from above, we find
\begin{align*}
L_f(q_1,t_2) &\leq L_f(q_1,t_2+d_{g(t_2)}(q_1,q_2)) +
Cd_{g(t_2)}(q_1,q_2)\\%
&\leq L_f(q_2,t_2) + Cd_{g(t_2)}(q_1,q_2),
\end{align*}
which proves Claim 3.
\end{proof}
The Lipschitz continuity in the time variable follows from Claim 1
and Claim 2. The Lipschitz continuity in the space variable follows
from Claim 3 and the symmetry between $q_1$ and $q_2$.\\\\%
From the definition of $L_f:M\times(0,T)\to\RR$, we see that it is
smooth outside of the set $\bigcup_t (C(t)\times\{t\})$, where for a
fixed time $t_1$ the \emph{cut locus} $C(t_1)$ is defined to be the
set of points $q\in M$ such that either there is more than one
minimal $\sL_f$-geodesic $\gamma:[0,t_1]\to M$ from $p=\gamma(0)$ to
$q=\gamma(t_1)$ or $q$ is conjugate to $p$ along $\gamma$. A point
$q$ is called conjugate to $p$ along $\gamma$ if there exists a
nontrivial $\sL_f$-Jacobi field $J$ along $\gamma$ with
$J(0)=J(t_1)=0$.\\\\%
As in the standard Riemannian geometry, the set $C_1(t_1)$ of
conjugate points to $(p,0)$ is contained in the set of critical
values for the $\sL_f$-exponential map from $(p,0)$ defined above.
Hence it has measure zero by Sard's theorem. If there exist more
than one minimal $\sL_f$-geodesic from $p$ to $q$, then $L(q,t_1)$
is not differentiable at $q$. But since $L_f(q,t_1)$ is Lipschitz,
it has to be differentiable almost everywhere by Rademacher's
theorem and thus the set $C_2(t_1)$ consisting of points for which
there exist more than one minimal $\sL_f$-geodesic also has to have
measure zero. Combining this, $C(t_1)=C_1(t_1)\cup C_2(t_1)$ has
measure zero for all $t_1\in (0,T)$ and so $\bigcup_t
(C(t)\times\{t\})$ is of measure zero, too. This finishes the proof
of the lemma.
\end{proof}
%
%
\section{Proof of Theorem \ref{1.mainthm}}\label{Sect5}
Making use of Lemma \ref{4.lemmaLip}, we first pretend that
$L_f(q,t_1)$ is smooth everywhere and derive formulas for $\abs{\D
L_f}^2$, $\dte L_f$ and $\Lap L_f$ under this assumption.
\begin{lemma}\label{5.lemma1}%
The reduced distance $L_f(q,t_1)$ has the gradient properties
\begin{align}
\abs{\nabla L_f(q,t_1)}^2 &= -4t_1S \hap \tfrac{4}{\sqrt{t_1}}K
+\tfrac{2}{\sqrt{t_1}}L_f(q,t_1),\label{5.eq1}\\%
\dte L_f(q,t_1) &= 2\sqrt{t_1}S \ham \tfrac{1}{t_1}K -
\tfrac{1}{2t_1}L_f(q,t_1),\label{5.eq2}
\end{align}
where
\begin{equation*}
K := \int_0^{t_1} t^{3/2}\sH(\sS, \ham X) dt
\end{equation*}
and $\sH(\sS,\ham X)$ is the Harnack type expression from Definition
\ref{1.defH}, evaluated at time $t$. Remember that in the backwards
case we interpret $\sH(\sS,X)$ as in (\ref{4.defH1}).
\end{lemma}
\begin{proof}
A minimizing curve satisfies $G_f(X)=0$, hence the first variation
formula above yields
\begin{equation*}
\delta_Y L_f(q,t_1) = 2\sqrt{t_1} \scal{X(t_1),Y(t_1)} =
\scal{\nabla L_f(q,t_1),Y(t_1)}.
\end{equation*}
Thus, the gradient of $L_f$ must be $\nabla L_f(q,t_1) = 2
\sqrt{t_1}X(t_1)$. This yields
\begin{equation}\label{5.eqNabla}%
\abs{\nabla L_f}^2 = 4t_1\abs{X}^2 = - 4t_1S +
4t_1\big(S+\abs{X}^2\big).
\end{equation}
Moreover, we compute
\begin{equation}\label{5.eqDt}
\begin{split}
\dte L_f(q,t_1) &=\tfrac{d}{dt_1}L_f(q,t_1) - \D_X L_f(q,t_1) =
\sqrt{t_1}\big(S+\abs{X}^2\big) - \scal{\nabla
L_f(q,t_1),X}\\%
&=\sqrt{t_1}\big(S+\abs{X}^2\big) - 2\sqrt{t_1}\abs{X}^2 =
2\sqrt{t_1}S - \sqrt{t_1} \big(S+\abs{X}^2\big).%
\end{split}
\end{equation}
Note that $\dte$ denotes the partial derivative with respect to
$t_1$ keeping the point $q$ fixed, while $\frac{d}{dt_1}$ refers to
differentiation along an $\sL$-geodesic, i.e.~simultaneously varying
the time $t_1$ and the point $q$. Next, we determine
$\big(S+\abs{X}^2\big)$ in terms of $L_f$. With the Euler-Lagrange
equation (\ref{4.eqGeo1}), we get
\begin{align*}
\tfrac{d}{dt} \big(S(\gamma(t))+\abs{X(t)}^2\big)
&=\dt S + \D_X S + 2\scal{\D_X X,X} \ham 2\sS(X,X)\\%
&=\dt S + 2\scal{\nabla S,X} - \tfrac{1}{t}\abs{X}^2 \hap
2\sS(X,X)\\%
&=\hap \sH(\sS,\ham X) - \tfrac{1}{t}\big(S+\abs{X}^2\big).
\end{align*}
From this we obtain
\begin{equation}\label{5.eqNotation}
t^{3/2}\tfrac{d}{dt}\big(S+\abs{X}^2\big) = \hap t^{3/2}\sH(\sS,\ham
X) - \sqrt{t}\big(S+\abs{X}^2\big)
\end{equation}
and thus by integrating and using the notation $K = \int_0^{t_1}
t^{3/2}\sH(\sS,\ham X) dt $, we conclude
\begin{align*}
\hap K - L_f(q,t_1) &=\int_0^{t_1} t ^{3/2}
\tfrac{d}{dt}\big(S+\abs{X}^2\big) dt\\%
&=t_1^{3/2} \big(S(\gamma(t_1))+\abs{X(t_1)}^2\big) - \int_0^{t_1}
\tfrac{3}{2}\sqrt{t}\big(S+\abs{X}^2\big) dt\\%
&=t_1^{3/2}\left(S+\abs{X}^2\right) - \tfrac{3}{2}L_f(q,t_1).
\end{align*}
Hence, we have%
\begin{equation}
t_1^{3/2}\big(S+\abs{X}^2\big) = \hap K + \tfrac{1}{2}L_f(q,t_1).
\end{equation}
If we insert this into (\ref{5.eqNabla}) and (\ref{5.eqDt}), we get
(\ref{5.eq1}) and (\ref{5.eq2}), respectively.
\end{proof}
To compute the second variation of $\sL_f(\gamma)$, we use the
following claim.
\paragraph{Claim 1:} Under the flow $\dt g_{ij} = \ham 2S_{ij}$, we have
\begin{equation}
\begin{split}
\dt \scal{\D_Y Y,X} &=\scal{\D_X \D_Y Y,X}+\scal{\D_Y Y,\D_X
X} \ham 2\sS(\D_Y Y,X)\\%
& \quad \ham 2(\D_Y \sS)(Y,X) \hap (\D_X \sS)(Y,Y).%
\end{split}
\end{equation}
\begin{proof}
We start with
\begin{equation}\label{5.eqClaim}
\dt \scal{\D_Y Y,X} =\scal{\D_X \D_Y Y,X}+\scal{\D_Y Y,\D_X X} \ham
2\sS(\D_Y Y,X) + \langle\dot{\D}_Y Y,X\rangle,
\end{equation}
where $\dot{\D} := \dt \D$. From \cite{Muller:Harnack}, page 21, we
know that under the flow $\dt g = h$, we have
\begin{equation*}
\langle\dot{\D}_U V,W\rangle = \tfrac{1}{2}(\D_U h)(V,W) -
\tfrac{1}{2}(\D_W h)(U,V) +\tfrac{1}{2} (\D_V h)(U,W).
\end{equation*}
Hence, with $U=V=Y$, $W=X$ and $h= \ham 2\sS$, we get
\begin{equation*}
\langle\dot{\D}_Y Y,X\rangle = \ham 2(\D_Y \sS)(Y,X) \hap (\D_X
\sS)(Y,Y).
\end{equation*}
Inserting this into (\ref{5.eqClaim}) proves the claim.
\end{proof}
Using Claim 1, we can now write $2\scal{\D_Y \D_X Y,X}$ as
\begin{align*}
2\scal{\D_Y \D_X Y,X} &=2\scal{\D_X \D_Y Y,X} +
2\scal{\Rm(Y,X)Y,X}\\%
&=2\dt \scal{\D_Y Y,X} - 2 \scal{\D_Y Y,\D_X X} \hap 4\sS(\D_Y Y,X)\\%
&\quad \hap 4(\D_Y \sS)(Y,X) \ham 2(\D_X \sS)(Y,Y) +
2\scal{\Rm(Y,X)Y,X},
\end{align*}
and a partial integration yields
\begin{equation}\label{5.eq9}
\begin{split}
\int_0^{t_1} 2\sqrt{t}\scal{\D_Y \D_X Y,X} dt &=2\sqrt{t}\scal{\D_Y
Y,X}\big{|}_0^{t_1} - \int_0^{t_1} \sqrt{t}\,\tfrac{1}{t}\scal{\D_Y Y,X} dt\\%
&\quad - \int_0^{t_1} 2\sqrt{t}\scal{\D_Y Y,\D_X X
\ham 2\sS(X,\cdot)} dt\\%
&\quad \hap \int_0^{t_1} \sqrt{t}\left(4(\D_Y \sS)(Y,X)
- 2(\D_X \sS)(Y,Y)\right)dt\\%
&\quad + \int_0^{t_1} 2\sqrt{t} \scal{\Rm(Y,X)Y,X}
dt.%
\end{split}
\end{equation}
If the geodesic equation (\ref{4.eqGeo1}) holds, we can write the
first two integrals on the right hand side of (\ref{5.eq9}) as
\begin{equation*}
-2\int_0^{t_1} \sqrt{t}\scal{\D_Y Y,\tfrac{1}{2t}X + \D_X X \ham
2\sS(X,\cdot)} dt  = -\int_0^{t_1}
\sqrt{t}\scal{\D_Y Y,\nabla S} dt,%
\end{equation*}
and equation (\ref{5.eq9}) becomes
\begin{equation}\label{5.eq10}
\begin{split}
\int_0^{t_1} 2\sqrt{t}\scal{\D_Y \D_X Y,X} dt &=2\sqrt{t}\scal{\D_Y
Y,X}\big{|}_0^{t_1} - \int_0^{t_1} \sqrt{t}\scal{\D_Y Y,\nabla S} dt\\%
&\quad \hap \int_0^{t_1} \sqrt{t}\left(4(\D_Y \sS)(Y,X)
- 2(\D_X \sS)(Y,Y)\right)dt\\%
&\quad + \int_0^{t_1} 2\sqrt{t} \scal{\Rm(Y,X)Y,X} dt.
\end{split}
\end{equation}
We can now compute the second variation of $\sL_f(\gamma)$ for
$\sL_f$-geodesics $\gamma$ where $G_f(X)=0$ is satisfied. Using the
first variation
\begin{equation*}
\delta_Y \sL_f(\gamma) = \int_{t _0}^{t_1} \sqrt{t}
(\D_Y S + 2\scal{\D_Y X,X}) dt%
\end{equation*}
from the last section, we compute
\begin{equation}\label{5.eq11}
\begin{split}
\delta_Y^2 \sL_f(\gamma) &=\int_0^{t_1} \sqrt{t} \big(\ds
\scal{\nabla
S,Y} + 2\scal{\D_Y \D_Y X,X}+2\abs{\D_Y X}^2\big) dt\\%
&=\int_0^{t_1} \sqrt{t} \big(\scal{\nabla S,\D_Y Y}
+\D_Y \D_Y S +2\abs{\D_X Y}^2 + 2\scal{\D_Y \D_X Y,X}\big) dt\\%
&=2\sqrt{t}\scal{\D_Y Y,X}\big{|}_0^{t_1} + \int_0^{t_1}
\sqrt{t}\left(\D_Y \D_Y S +2\abs{\D_X Y}^2\right) dt\\%
&\quad \ham \int_0^{t_1} \sqrt{t}(2(\D_X \sS)(Y,Y)- 4(\D_Y
\sS)(Y,X)) dt\\%
&\quad + \int_0^{t_1} 2\sqrt{t}\scal{\Rm(Y,X)Y,X} dt,
\end{split}
\end{equation}
where we used (\ref{5.eq10}) in the last step. Now choose the test
variation $Y(t)$ such that
\begin{equation}\label{5.eqODE}%
\D_X Y = \hap \sS(Y,\cdot) + \tfrac{1}{2t}Y,
\end{equation}
which implies $\dt\abs{Y}^2 = \ham 2 \sS(Y,Y) + 2\scal{\D_X Y,Y} =
\frac{1}{t}\abs{Y}^2$ and hence $\abs{Y(t)}^2 = t  / t_1$, in
particular $Y(0)=0$. We have
\begin{equation}\label{5.eqHess}
\begin{split}
\Hess_{L_f}(Y,Y) &=\D_Y \D_Y L_f = \delta_Y^2(L_f) -
\scal{\D_Y Y,\nabla L_f}\\%
&\leq \delta_Y^2{\sL_f} - 2\sqrt{t_1}\scal{\D_Y Y,X}(t_1),
\end{split}
\end{equation}
where the $Y$ in $\Hess_{L_f}(Y,Y) = \D_Y \D_Y L_f$ denotes a vector
$Y(t_1) \in T_q M$, while in $\delta_Y^2\sL_f$ it denotes the
associated variation of the curve, i.e.~the vector field $Y(t)$
along $\gamma$ which solves the above ODE (\ref{5.eqODE}). Note that
(\ref{5.eqHess}) holds with equality if $Y$ is an $\sL_f$-Jacobi
field. We obtain
\begin{equation}\label{5.eq14}
\begin{split}
\Hess_{L_f}(Y,Y) &\leq \int_0^{t_1} \sqrt{t}\big(\D_Y \D_Y
S +2\abs{\D_X Y}^2 + 2\scal{\Rm(Y,X)Y,X}\big) dt\\%
&\quad \ham \int_0^{t_1} \sqrt{t}(2(\D_X \sS)(Y,Y)- 4(\D_Y
\sS)(Y,X)) dt.
\end{split}
\end{equation}
\begin{lemma}\label{5.lemma2}%
For $K$ defined as in Lemma \ref{5.lemma1}, and under the assumption
$\sD(\sS,Z)\geq 0, \; \forall Z \in \Gamma(TM)$, the distance
function $L_f(q,t_1)$ satisfies
\begin{equation}\label{5.eq15}%
\Lap L_f(q,t_1) \leq \frac{n}{\sqrt{t_1}} \hap 2\sqrt{t_1}S -
\frac{1}{t_1}K.
\end{equation}
\end{lemma}
\begin{proof}
Note that with (\ref{5.eqODE}) we find
\begin{equation}\label{5.eq16}
\begin{split}
\abs{\D_X Y}^2 &=\abs{\sS(Y,\cdot)}^2 \hap \tfrac{1}{t}\sS(Y,Y)
+ \tfrac{1}{4t ^2}\abs{Y(t)}^2\\%
&=\abs{\sS(Y,\cdot)}^2 \hap \tfrac{1}{t}\sS(Y,Y) +
\tfrac{1}{4t\,t_1},%
\end{split}
\end{equation}
as well as
\begin{equation}\label{5.eq17}
\begin{split}
\tfrac{d}{dt}\sS(Y(t),Y(t)) &=(\dt \sS)(Y,Y) + (\D_X
\sS)(Y,Y) + 2\sS(\D_X Y,Y)\\%
&=(\dt \sS)(Y,Y) + (\D_X \sS)(Y,Y)+ \tfrac{1}{t}\sS(Y,Y) \hap
2\abs{\sS(Y,\cdot)}^2.
\end{split}
\end{equation}
Using (\ref{5.eq16}), a partial integration and then (\ref{5.eq17}),
we get from (\ref{5.eq14})
\begin{align*}
\Hess_{L_f}(Y,Y) &\leq \int_0^{t_1} \sqrt{t}\big(\D_Y \D_Y
S + 2\scal{\Rm(Y,X)Y,X}\big) dt\\%
&\quad \ham \int_0^{t_1} \sqrt{t}\big(2(\D_X \sS)(Y,Y)-
4(\D_Y \sS)(Y,X)\big) dt\\%
&\quad + \int_0^{t_1} \sqrt{t}\big(2\abs{\sS(Y,\cdot)}^2 \hap
\tfrac{2}{t}\sS(Y,Y) + \tfrac{1}{2t\,t_1}\big) dt\\%
&= \frac{1}{\sqrt{t_1}}-\int_0^{t_1} \sqrt{t}\,\sH(\sS,\ham X,Y) dt
\hap \int_0^{t_1} \sqrt{t}\big(2(\D_X
\sS)(Y,Y) \hap 4\abs{\sS(Y,\cdot)}^2\big) dt\\%
&\quad \hap \int_0^{t_1} \sqrt{t}\big(\tfrac{3}{t}\sS(Y,Y)+
2(\dt\sS)(Y,Y)\big) dt\\%
&=\frac{1}{\sqrt{t_1}} -\int_0^{t_1} \sqrt{t}\,\sH(\sS,\ham X,Y)dt
\hap \int_0^{t_1} \sqrt{t}\big(2\tfrac{d}{dt}\sS(Y,Y)
+ \tfrac{1}{t}\sS(Y,Y)\big) dt\\%
&=\frac{1}{\sqrt{t_1}} \hap 2\sqrt{t_1}\sS(Y,Y) - \int_0^{t_1}
\sqrt{t}\,\sH(\sS,\ham X,Y) dt.
\end{align*}
Here, $\sH(\sS,\ham X,Y)$ denotes the Harnack type expression from
Definition \ref{1.defH} evaluated at time $t$. Remember that in the
backwards case $\sH(\sS,X,Y)$ has to be interpreted as in
(\ref{4.defH2}). Now let $\{Y_i(t_1)\}$ be an orthonormal basis of
$T_q M$, and define $Y_i(t)$ as above, solving the ODE
(\ref{5.eqODE}). We compute
\begin{align*}
\dt \scal{Y_i,Y_j} &= \ham 2\sS(Y_i,Y_j) + \scal{\D_X Y_i,Y_j} +
\scal{Y_i,\D_X Y_j}\\%
&= \ham 2\sS(Y_i,Y_j) + \scal{\hap \sS(Y_i,\cdot) +
\tfrac{1}{2t}Y_i,Y_j} + \scal{Y_i,\hap \sS(Y_j,\cdot) +
\tfrac{1}{2t}Y_j}\\%
&= \tfrac{1}{t}\scal{Y_i,Y_j}.
\end{align*}
Thus the $\{Y_i(t)\}$ are orthogonal with $\scal{Y_i(t),Y_j(t)} =
\frac{t}{t_1}\scal{Y_i(t_1),Y_j(t_1)} = \frac{t}{t_1}\delta_{ij}$.
In particular, there exist orthonormal vector fields $e_i(t)$ along
$\gamma$ with $Y_i(t) = \sqrt{t/t_1}\,e_i(t)$. Summing over
$\{e_i\}$ yields
\begin{align*}
\Lap L_f(q,t_1) &\leq \sum_i\Big(\frac{1}{\sqrt{t_1}} \hap
2\sqrt{t_1}\sS(Y_i,Y_i) - \int_0^{t_1}
\sqrt{t}\,\sH(\sS,\ham X,Y_i) dt\Big)\\%
&=\frac{n}{\sqrt{t_1}} \hap 2\sqrt{t_1}S - \frac{1}{t_1}
\int_0^{t_1} t^{3/2} \sum_i \sH(\sS,\ham X,e_i) dt\\%
&=\frac{n}{\sqrt{t_1}} \hap 2\sqrt{t_1}S -
\frac{1}{t_1} \int_0^{t_1} t^{3/2}\big(\sH(\sS,\ham X)+\sD(\sS,\ham X)\big)dt\\%
&\leq\frac{n}{\sqrt{t_1}} \hap 2\sqrt{t_1}S - \frac{1}{t_1} K,%
\end{align*}
using Lemma \ref{1.lemmaD} and the assumption $\sD(\sS,\ham X)\geq 0$.%
\end{proof}
The three formulas from Lemma \ref{5.lemma1} and Lemma
\ref{5.lemma2} can now be combined to one evolution inequality for
the reduced distance function $\ell_f(q,t_1) =
\frac{1}{2\sqrt{t_1}}L_f(q,t_1)$. From (\ref{5.eq1}), (\ref{5.eq2})
and (\ref{5.eq15}), we get
\begin{align*}
\abs{\D\ell_f}^2 &= \frac{1}{4t_1}\abs{\D L_f}^2 = -S +
\frac{1}{t_1}\ell_f \hap \frac{1}{t_1^{3/2}}K,\\%
\dte \ell_f &= -\frac{1}{4t_1^{3/2}}L_f + \frac{1}{2\sqrt{t_1}}\dte
L_f= -\frac{1}{t_1}\ell_f + S \ham \frac{1}{2t_1^{3/2}}K,\\%
\Lap \ell_f &= \frac{1}{2\sqrt{t_1}}\Lap L_f \leq \frac{n}{2t_1}
\hap S - \frac{1}{2t_1^{3/2}}K,
\end{align*}
and thus
\begin{equation}\label{5.eqTogether}
\Lap \ell_f \hap \dte \ell_f \hap \abs{\D\ell_f}^2 \ham S -
\tfrac{n}{2t} \leq 0.
\end{equation}
This is equivalent to
\begin{equation}\label{5.eqHeateq}
(\dt \hap\, \Lap \ham S)v_f(q,t) \leq 0,
\end{equation}
where $v_f(q,t):= (4\pi t)^{-n/2}e^{\hap \ell_f(q,t)}$ is the
density function for the reduced volume $V_f(t)$.\\\\%
Note that so far we pretended that $L_f(q,t_1)$ is smooth. In the
general case, it is obvious that the inequality (\ref{5.eqTogether})
holds in the classical sense at all points where $L_f$ is smooth.
But what happens at all the other points? This question is answered
by the following lemma.
\begin{lemma}
The inequality (\ref{5.eqTogether}) holds on $M\times (0,T)$ in the
\emph{barrier sense}, i.e.~for all $(q_*,t_*)\in M\times(0,T)$ there
exists a neighborhood $U$ of $q_*$ in $M$, some $\eps > 0$ and a
smooth upper barrier $\tilde{\ell}_f$ defined on
$U\times(t_*-\eps,t_*+\eps)$ with $\tilde{\ell}_f \geq \ell_f$ and
$\tilde{\ell}_f(q_*,t_*)=\ell_f(q_*,t_*)$ which satisfies
(\ref{5.eqTogether}). Moreover, (\ref{5.eqTogether}) holds on
$M\times(0,T)$ in the \emph{distributional sense}.
\end{lemma}
\begin{proof}
Given $(q_*,t_*)\in M\times(0,T)$, let $\gamma:[0,t_*]\to M$ be a
minimal $\sL_f$-geodesic from $p$ to $q_*$, so that
$\ell_f(q_*,t_*)=\frac{1}{2\sqrt{t_*}}\sL_f(\gamma)$. Extend
$\gamma$ to a smooth $\sL_f$-geodesic $\gamma:[0,t_*+\eps]\to M$ for
some $\eps>0$. For a given orthonormal basis $\{Y_i(t_*)\}$ of
$T_{q_*}M$, solve the ODE (\ref{5.eqODE}) on $[0,t_*+\eps]$ and let
$\gamma_i(s,t)$ be a variation of $\gamma(t)$ in the direction of
$Y_i$, i.e.~$\gamma_i(0,t)=\gamma(t)$ and $\ds \gamma_i(s,t)|_{s=0}
=Y_i(t)$. Finally, for a small neighborhood $U$ of $q_*$ we choose a
smooth family of curves $\eta_{q,t_1}:[0,t_1]\to M$ from
$\eta_{q,t_1}(0)=p$ to $\eta_{q,t_1}(t_1)=q \in U$,
$t_1\in(t_*-\eps,t_*+\eps)$, with the following property:
\begin{equation*}
\eta_{\gamma_i(s,t),t}=\gamma_i(s,\cdot) |_{[0,t]}, \quad \forall
t\in (t_*-\eps,t_*+\eps) \textrm{ and } \abs{s}< \eps.
\end{equation*}
Define $\tilde{L}_f(q,t_1):=\sL_f(\eta_{q,t_1})$ and
$\tilde{\ell}_f(q,t_1)=\frac{1}{2\sqrt{t_1}}\tilde{L}_f(q,t_1)$. By
construction $\eta_{q_*,t_*}=\gamma |_{[0,t_*]}$ and hence
$\tilde{L}_f(q,t_1)$ is a smooth upper barrier for $L_f(q,t_1)$ with
$\tilde{L}_f(q_*,t_*)=L_f(q_*,t_*)$. Moreover, $\tilde{L}_f$
satisfies the formulas in Lemma \ref{5.lemma1} and Lemma
\ref{5.lemma2}. Thus $\tilde{\ell}_f(q,t_1)$ is a smooth upper
barrier for $\ell_f(q,t_1)$ that satisfies (\ref{5.eqTogether}).\\\\%
To see that (\ref{5.eqTogether}) holds in the distributional sense,
we use the general fact that if a differential inequality of the
type (\ref{5.eq15}) holds in the barrier sense and we have a bound
on $\abs{\D L_f}$, then the inequality also holds in the
distributional sense, see for example \cite{RF:TAI}, Lemma 7.125.
Obviously (\ref{5.eq1}) and (\ref{5.eq2}) also hold in the
distributional sense, since they hold in the barrier sense.
Combining this, the claim from the lemma follows.
\end{proof}
\begin{proof}[Proof of Theorem \ref{1.mainthm}]
Since (\ref{5.eqTogether}) and hence also (\ref{5.eqHeateq}) hold in
the distributional sense, we simply compute, using $\dt dV = \ham S
dV$,
\begin{equation}
\begin{split}
\dt V_f(t) &= \int_M v_f(q,t) \dt dV + \int_M \dt v_f(q,t)dV\\%
&\leq \int v_f(q,t) \cdot (\ham S)dV \hap \int_M (S-\Lap)v_f(q,t) dV\\%
&= \ham \int_M \Lap v_f(q,t) dV = 0.
\end{split}
\end{equation}
Thus, the reduced volume $V_f(t)$ is non-increasing in $t$.
\end{proof}

\makeatletter
\def\@listi{%
  \itemsep=0pt
  \parsep=1pt
  \topsep=1pt}
\makeatother
{\fontsize{10}{11}\selectfont

\vspace{10mm}

Reto M\"uller\\
{\sc Scuola Normale Superiore di Pisa, 56126 Pisa, Italy}
\end{document}